\documentclass[11pt,reqno]{amsart}

\usepackage{amssymb}
\usepackage{amsmath}
\usepackage{enumerate}
\usepackage{bbm}
\usepackage{esint}
\usepackage{a4wide}

\usepackage{bookmark}
\usepackage{hyperref}
\hypersetup{pdfstartview={FitH}}

\newtheorem{theorem}{Theorem}
\newtheorem{proposition}[theorem]{Proposition}
\newtheorem{lemma}[theorem]{Lemma}

\DeclareFontFamily{U}{mathx}{\hyphenchar\font45}
\DeclareFontShape{U}{mathx}{m}{n}{
<5> <6> <7> <8> <9> <10>
<10.95> <12> <14.4> <17.28> <20.74> <24.88>
mathx10
}{}
\DeclareSymbolFont{mathx}{U}{mathx}{m}{n}
\DeclareFontSubstitution{U}{mathx}{m}{n}
\DeclareMathAccent{\widecheck}{0}{mathx}{"71}

\numberwithin{equation}{section}

\begin{document}

\title[On side lengths of corners in positive density subsets]{On side lengths of corners in positive density subsets of the Euclidean space}

\author[P. Durcik]{Polona Durcik}
\address{Polona Durcik, Universit\"at Bonn, Endenicher Allee 60, 53115 Bonn, Germany}
\email{durcik@math.uni-bonn.de}

\author[V. Kova\v{c}]{Vjekoslav Kova\v{c}}
\address{Vjekoslav Kova\v{c}, Department of Mathematics, Faculty of Science, University of Zagreb, Bijeni\v{c}ka cesta 30, 10000 Zagreb, Croatia}
\email{vjekovac@math.hr}

\author[L. Rimani\'{c}]{Luka Rimani\'{c}}
\address{Luka Rimani\'{c}, School of Mathematics, University of Bristol, University Walk, Bristol BS8 1TW, United Kingdom}
\email{luka.rimanic@bristol.ac.uk}

\date{\today}


\begin{abstract}
We generalize a result by Cook, Magyar, and Pramanik \cite{CMP15} on three-term arithmetic progressions in subsets of $\mathbb{R}^d$ to corners in subsets of $\mathbb{R}^d\times\mathbb{R}^d$. More precisely, if $1<p<\infty$, $p\neq 2$, and $d$ is large enough, we show that an arbitrary measurable set $A\subseteq\mathbb{R}^d\times\mathbb{R}^d$ of positive upper Banach density contains corners $(x,y)$, $(x+s,y)$, $(x,y+s)$ such that the $\ell^p$-norm of the side $s$ attains all sufficiently large real values. Even though we closely follow the basic steps from \cite{CMP15}, the proof diverges at the part relying on harmonic analysis. We need to apply a higher-dimensional variant of a multilinear estimate from \cite{DKST16}, which we establish using the techniques from \cite{DKST16} and \cite{DKT16}.
\end{abstract}

\maketitle


\section{Introduction}
The \emph{upper Banach density} of a set $A \subseteq \mathbb{R}^d$ is defined as
\[
\overline{\delta}_d(A) := \limsup_{N\rightarrow \infty} \sup_{x\in\mathbb{R}^d} \frac{\left|A\cap (x+[0,N]^d)\right|}{\left|x+[0,N]^d\right|},
\]
where $|\cdot|$ denotes the $d$-dimensional Lebesgue measure, so that $|x+[0,N]^d| = N^d$.
If $d\geq 2$ and $\overline{\delta}_d(A)>0$, then there exists a sufficiently large $\lambda_0(A)>0$ such that for any real number $\lambda\geq\lambda_0(A)$ the set $A$ contains points $x$ and $x+s$ with $\|s\|_{\ell^2}=\lambda$. This fact was shown independently by Bourgain \cite{Bou86}, Falconer and Marstrand \cite{FM86}, and Furstenberg, Katznelson, and Weiss \cite{FKW90}. Here $\|\cdot\|_{\ell^2}$ denotes the Euclidean norm. More generally, we denote the \emph{$\ell^p$-norm} on $\mathbb{R}^d$ by
\[
\|s\|_{\ell^p} := \begin{cases}
\Big(\sum_{i=1}^d\limits |s_i|^p\Big)^{1/p} & \text{for } 1\leq p<\infty, \\
\max_{1\leq i\leq d}\limits |s_i| & \text{for } p=\infty,
\end{cases}
\]
if $s=(s_1,\ldots,s_d)$. It is another observation by Bourgain \cite{Bou86} that the same statement fails if we replace the trivial pattern $x$, $x+s$ by a \emph{$3$-term arithmetic progression}
\[
x,\ x+s,\ x+2s.
\]
Indeed, the set $A$ obtained as a union of the annuli $n-1/10 \leq \|x\|_{\ell^2}^2 \leq n+1/10$ as $n$ runs over the positive integers clearly has density $\overline{\delta}_d(A)>0$, but if $x,s\in\mathbb{R}^d$ are such that $x$, $x+s$, $x+2s\in A$, then the parallelogram law
\[
2\|x+s\|_{\ell^2}^2 + 2\|s\|_{\ell^2}^2 = \|x+2s\|_{\ell^2}^2 + \|x\|_{\ell^2}^2
\]
implies that $n-2/5 \leq 2\|s\|_{\ell^2}^2 \leq n+2/5$ for some integer $n$. Therefore, the $\ell^2$-norms of the common differences $s$ of the $3$-term progressions in $A$ cannot attain values in the set
\[
\bigcup_{n=1}^{\infty} \Big(\sqrt{\frac{5n-3}{10}},\sqrt{\frac{5n-2}{10}}\Big),
\]
which contains arbitrarily large numbers.

An interesting phenomenon occurs in large dimensions if one replaces the $\ell^2$-norm by other $\ell^p$-norms. A recent result by Cook, Magyar, and Pramanik \cite{CMP15} sheds new light on the Euclidean density theorems by establishing that a set of positive upper Banach density still contains $3$-term arithmetic progressions such that the $\ell^p$-norms of their common differences attain all sufficiently large values when $1<p<\infty$ and $p\neq 2$.

\begin{theorem}[from \cite{CMP15}]\label{thm:cmptheorem}
For any $p\in(1,2)\cup(2,\infty)$ there exists $d_p \geq 2$ such that for every integer $d\geq d_p$ the following holds. For any measurable set $A\subseteq\mathbb{R}^d$ satisfying $\overline{\delta}_d(A)>0$ one can find $\lambda_0(A)>0$ having the property that for any real number $\lambda\geq\lambda_0(A)$, there exist $x,s\in\mathbb{R}^d$ such that $x$, $x+s$, $x+2s\in A$ and $\|s\|_{\ell^p}=\lambda$.
\end{theorem}

The authors of \cite{CMP15} place this result in the context of the Euclidean Ramsey theory and demonstrate that it is sharp with regard to the exponent $p$. Indeed, measuring the common differences in the $\ell^1$ or the $\ell^\infty$-norm allows for quite straightforward counterexamples. They only leave the optimal value of the dimension threshold $d_p$ as an open problem.

The aim of this paper is a generalization of Theorem~\ref{thm:cmptheorem} to so-called \emph{corners}, which are patterns in $\mathbb{R}^d\times\mathbb{R}^d$ of the form
\begin{equation}\label{eq:defcorner}
(x,y),\ (x+s,y),\ (x,y+s)
\end{equation}
for some $x,y,s\in\mathbb{R}^d$, $s\neq 0$. The fact that any subset of $\mathbb{Z}\times\mathbb{Z}$ of positive upper density contains a corner was proved by Ajtai and Szemer\'{e}di \cite{AS74}, while the first ``reasonable'' quantitative upper bounds (of the form $n^2/(\log\log n)^c$ with $c>0$) for subsets of $\{1,\ldots,n\}\times\{1,\ldots,n\}$ without corners are due to Shkredov \cite{Shk06a}, \cite{Shk06b}.

We are interested in finding corners exhibiting all sufficiently large side lengths in positive upper Banach density subsets of $\mathbb{R}^d\times\mathbb{R}^d$. Here is the main result of this paper.

\begin{theorem}\label{thm:cornerstheorem}
For any $p\in(1,2)\cup(2,\infty)$ there exists $d_p \geq 2$ such that for every integer $d\geq d_p$ the following holds. For any measurable set $A\subseteq\mathbb{R}^d\times\mathbb{R}^d$ satisfying $\overline{\delta}_{2d}(A)>0$ one can find $\lambda_0(A)>0$ with the property that for any real number $\lambda\geq\lambda_0(A)$, there exist $x,y,s\in\mathbb{R}^d$ such that $(x,y)$, $(x+s,y)$, $(x,y+s)\in A$ and $\|s\|_{\ell^p}=\lambda$.
\end{theorem}

It is easy to see that Theorem~\ref{thm:cornerstheorem} implies Theorem~\ref{thm:cmptheorem}. One simply observes that if $A\subseteq\mathbb{R}^d$ has $\overline{\delta}_d(A)>0$, then the set $\widetilde{A}$ defined by
\[
\widetilde{A} := \{(x,y)\in\mathbb{R}^d\times\mathbb{R}^d : y-x\in A\}
\]
satisfies $\overline{\delta}_{2d}(\widetilde{A})>0$.
For this purpose it is convenient to change the coordinates on $\mathbb{R}^d\times\mathbb{R}^d$ to $(x',y')=(x+y,y-x)/\sqrt{2}$ and rotate the cubes $[0,N]^{2d}$ in the definition of $\overline{\delta}_{2d}(\widetilde{A})$, possibly at the cost of losing a multiplicative constant.
Moreover, any corner in $\widetilde{A}$ with side $s$ via the projection $(x,y)\mapsto y-x$ gives rise to a $3$-term arithmetic progression in $A$ with $s$ as its common difference. The same argument also enables the use of the previously mentioned counterexamples, which rule out the possibility of Theorem~\ref{thm:cornerstheorem} holding for $p=1$, $2$, or $\infty$.

We need to emphasize that our proof of Theorem~\ref{thm:cornerstheorem} closely follows the outline of \cite{CMP15}. The most significant novelty appears in the harmonic analysis part of the proof, where we need to prove an estimate for certain ``entangled'' singular multilinear forms, stated as Theorem~\ref{thm:analytictheorem} below. For previous work on patterns in sufficiently dense subsets of the Euclidean space we refer for instance to \cite{Bou86}, \cite{HLM16}, and \cite{LM16}. Bourgain \cite{Bou86} has shown that any set of positive upper Banach density in $\mathbb{R}^k$ contains isometric copies of all sufficiently large dilates of a fixed non-degenerate $k$-point (i.e.\@ $(k-1)$-dimensional) simplex; non-degeneracy being essential there. Moreover, Lyall and Magyar \cite{LM16} extended his result to Cartesian products of two non-degenerate simplices. In particular, they are able to detect patterns like $(x,y)$, $(x+s,y)$, $(x,y+t)$ with $\|s\|_{\ell^2}=\|t\|_{\ell^2}$, which have more degrees of freedom than the corners in definition \eqref{eq:defcorner}, and as such are easier to handle.

We now turn to the analytical ingredients that will be needed in the proof of Theorem~\ref{thm:cornerstheorem}.
For $1\leq p<\infty$ let $\|\cdot\|_{\textup{L}^p}$ denote the \emph{Lebesgue $\textup{L}^p$-norm} defined by
\[
\|f\|_{\textup{L}^p} := \Big(\int_{\mathbb{R}^d} |f(x)|^p dx\Big)^{1/p}
\]
and let $\textup{L}^p(\mathbb{R}^d)$ be the corresponding Banach space of a.e.-classes of measurable functions $f$ such that $\|f\|_{\textup{L}^p}<\infty$.
Denote by $\partial^\kappa f := \partial_{1}^{\kappa_1}\cdots\partial_{d}^{\kappa_d}f$ the partial derivative of a function $f\colon\mathbb{R}^d\to\mathbb{C}$ with respect to the multi-index $\kappa=(\kappa_1,\ldots,\kappa_d)$, the order of which will be written $|\kappa|:=\kappa_1+\cdots+\kappa_d$.
Finally, we use the notation $\widehat{f}$ and $\widecheck{f}$ for the Fourier transform and its inverse respectively, both initially defined for Schwartz functions $f$ by \eqref{eq:ft} and \eqref{eq:ift} below, and then extended to tempered distributions.

\begin{theorem}\label{thm:analytictheorem}
Suppose that $m\in\textup{C}^{\infty}(\mathbb{R}^{2d})$ satisfies the standard symbol estimates, i.e.\@ for any multi-index $\kappa$ there exists a constant $C_\kappa\in[0,\infty)$ such that
\begin{equation}\label{eq:symbolest}
\big| (\partial^\kappa m)(\xi,\eta) \big| \leq C_\kappa \|(\xi,\eta)\|_{\ell^2}^{-|\kappa|}
\end{equation}
for all $(\xi,\eta)\in\mathbb{R}^d\times\mathbb{R}^d$, $(\xi,\eta)\neq(0,0)$.
Suppose also that the tempered distribution $K=\widecheck{m}$ is equal to a bounded compactly supported function (denoted by the same letter).
Then for any real-valued $F,G\in\textup{L}^4(\mathbb{R}^{2d})$ we have the estimate
\[
\Big| \int_{(\mathbb{R}^d)^4} F(x+u,y) G(x,y+u) F(x+v,y) G(x,y+v) K(u,v) du dv dx dy \Big| \leq C \|F\|_{\textup{L}^4}^2 \|G\|_{\textup{L}^4}^2,
\]
with a constant $C\in[0,\infty)$ depending only on the dimension $d$ and the constants $C_\kappa$.
\end{theorem}

We will only need a particular case of the theorem when $F=G$, but the given formulation is more natural since the proof will perform different changes of variables in $F$ and $G$.

The singular integral form in Theorem~\ref{thm:analytictheorem} will appear by expanding out a certain square function quantity; see the proof of Proposition~\ref{prop:ebounded} below. It is more singular than the form used in \cite{CMP15} for the same purpose, so we cannot invoke any standard references on modulation-invariant operators. In fact, boundedness of a related singular integral operator, defined as
\begin{equation}\label{eq:introtht}
T(F,G)(x,y) := \textup{p.v.}\int_{\mathbb{R}} F(x+u,y) G(x,y+u) \frac{du}{u}, \quad (x,y)\in\mathbb{R}^2,
\end{equation}
and called the \emph{triangular Hilbert transform}, is currently an open problem; see \cite{KTZ15} for the partial results.

Only recently the techniques required for bounding the form in Theorem~\ref{thm:analytictheorem} were developed as byproducts of the papers \cite{DKST16} and \cite{DKT16}, both of which are primarily concerned with unrelated problems.
Indeed, Theorem~\ref{thm:analytictheorem} can be viewed as a higher-dimensional variant of an auxiliary estimate from \cite{DKST16}, which established a norm-variation bound
\begin{equation}\label{eq:introvar}
\sup_{0<t_0<t_1<\cdots<t_m} \sum_{j=1}^{m} \|A_{t_j}(F,G)-A_{t_{j-1}}(F,G)\|_{\textup{L}^2}^2 \leq C \|F\|_{\textup{L}^4}^2 \|G\|_{\textup{L}^4}^2
\end{equation}
for two-dimensional bilinear averages
\[
A_t(F,G)(x,y) := \frac{1}{t} \int_{0}^{t} F(x+u,y) G(x,y+u) du, \quad (x,y)\in\mathbb{R}^2.
\]
Inequality \eqref{eq:introvar} in turn proved a quantitative result on the convergence of ergodic averages with respect to two commuting transformations.
Moreover, the paper \cite{DKT16} studied multilinear analogs of \eqref{eq:introtht}, with a more modest goal of proving boundedness with constants growing like $(\log(R/r))^{1-\epsilon}$ as $R/r\to\infty$, where the integration variable $u$ is now restricted to intervals $[-R,-r]$ and $[r,R]$ for $0<r<R$.
Interestingly, early instances of the method used for solving these problems were devised for bounding significantly less singular variants of the operator \eqref{eq:introtht}, such as
\[
T(F,G)(x,y) := \textup{p.v.}\int_{\mathbb{R}^2} F(x+u,y) G(x,y+v) K(u,v) du dv, \quad (x,y)\in\mathbb{R}^2;
\]
see \cite{D15} and \cite{K12}. Roughly speaking, the mentioned technique can be described as follows. Instead of decomposing the given operator and bounding its pieces, one rather performs a structural induction and gradually symmetrizes it by repeated applications of the Cauchy-Schwarz inequality and an integration by parts identity. Eventually, the operator in question becomes so symmetric that a monotonicity argument applies, bounding it simply by single-scale objects.

Finally, let us say a few words about the organization of this paper. In Section~\ref{sec:hapart} we give a detailed self-contained proof of Theorem~\ref{thm:analytictheorem}. Unlike in \cite{DKST16}, where the proof of a special case was given, we do not need any finer control of the constant $C$ here, and are able to make use of further ideas from \cite{DKT16}. Section~\ref{sec:combpart} contains the predominantly combinatorial part of the proof: we derive Theorem~\ref{thm:cornerstheorem} from Theorem~\ref{thm:analytictheorem} by mimicking the steps from \cite{CMP15}. Consequently, we frequently refer to \cite{CMP15} and only comment on the ingredients that have to be altered. Finally, in Section~\ref{sec:genremarks} we discuss the current obstructions to extending Theorem~\ref{thm:cmptheorem} to longer progressions and Theorem~\ref{thm:cornerstheorem} to generalized corners.


\section{The analytical part: Proof of Theorem~\ref{thm:analytictheorem}}
\label{sec:hapart}
If $A$ and $B$ are two nonnegative quantities, then $A\lesssim_P B$ will denote the inequality $A\leq C B$, with some finite constant $C$ depending on a set of parameters $P$. We will write $A\sim_P B$ if both $A\lesssim_P B$ and $B\lesssim_P A$ hold. The standard inner product on $\mathbb{R}^d$ will be written $(x,y)\mapsto x\cdot y$, while the Euclidean norm $\|\cdot\|_{\ell^2}$ will simply be denoted by $\|\cdot\|$ in this section. Moreover, let $\mathcal{S}(\mathbb{R}^d)$ be the Schwartz space on $\mathbb{R}^d$ and let $\mathbbm{i}$ denote the imaginary unit.
We normalize the \emph{Fourier transform} of a $d$-dimensional Schwartz function $f$ as in
\begin{equation}\label{eq:ft}
\widehat{f}(\xi) := \int_{\mathbb{R}^d} f(x) e^{-2\pi \mathbbm{i} x\cdot\xi} dx,
\end{equation}
so that the \emph{inverse Fourier transform} is given by the formula
\begin{equation}\label{eq:ift}
\widecheck{f}(x) = \int_{\mathbb{R}^d} f(\xi) e^{2\pi \mathbbm{i} x\cdot\xi} d\xi.
\end{equation}

Throughout this section we will use the following notation for the standard Gaussian function on $\mathbb{R}^d$ and its partial derivatives:
\begin{align*}
g(x) & := e^{-\pi \|x\|^2}, \\
h^i (x) & := \partial_i g (x)\quad \text{for } i=1,\ldots,d.
\end{align*}
Moreover, for a function $f\colon\mathbb{R}^d\to\mathbb{C}$ we will denote by $f_t$ its $\textup{L}^1$-normalized dilate by $t>0$, defined as
\begin{equation}\label{eq:normalizationdef}
f_t(x):=t^{-d}f(t^{-1}x).
\end{equation}
An important property of the Fourier transform is $\widehat{f_t}(\xi)=\widehat{f}(t\xi)$.

We begin by stating an ``integration by parts'' lemma, which will be used several times in the proof of Theorem~\ref{thm:analytictheorem}.
Its one-dimensional variant can be found in \cite{D15} or \cite{DKST16}, but we prefer to give a self-contained proof.
For real-valued functions $\psi,\varphi\in\mathcal{S}(\mathbb{R}^d)$ and $F\in\mathcal{S}(\mathbb{R}^{2d})$ we define the singular integral form
\begin{align}
\Theta_{\psi,\varphi}(F) := & \int_0^\infty \int_{(\mathbb{R}^d)^6} F(x,x') F(x,y') F(y,x') F(y,y') \nonumber \\
& \psi_t(x-q)\psi_t(y-q) \varphi_t(x'-p)\varphi_t(y'-p) dx dy dx' dy' dp dq \frac{dt}{t}. \label{eq:defoftheta}
\end{align}
Note that
$\Theta_{\psi,\varphi}(F)$ can be rewritten as
\begin{align*}
& \int_0^\infty \int_{(\mathbb{R}^d)^4} \Big( \int_{\mathbb{R}^d} F(x,x') F(x,y') \psi_t(x-q) dx \Big)^2
\varphi_t(x'-p) \varphi_t(y'-p) dx' dy' dp dq \frac{dt}{t} \\
& = \int_0^\infty \int_{(\mathbb{R}^d)^4} \Big( \int_{\mathbb{R}^d} F(x,x') F(y,x') \varphi_t(x'-p) dx' \Big)^2
\psi_t(x-q) \psi_t(y-q) dx dy dp dq \frac{dt}{t},
\end{align*}
so that $\Theta_{\psi,\varphi}(F)\geq 0$ when $\varphi\geq 0$ or $\psi\geq 0$.

\begin{lemma}\label{lemma:mainlemma}
For any real-valued function $F\in\mathcal{S}(\mathbb{R}^{2d})$ and any $\alpha,\beta >0$ we have the estimate
\[
\sum_{i=1}^d\Theta_{h_\alpha^i,g_\beta}(F) \lesssim \|F\|^4_{\textup{L}^4} ,
\]
where the implicit constant is an absolute one, i.e.\@ independent of $\alpha$, $\beta$, $d$, and $F$.
\end{lemma}

\begin{proof}[Proof of Lemma~\ref{lemma:mainlemma}]
We claim that
\begin{equation}\label{eq:telescoping}
\sum_{i=1}^d \big( \Theta_{h^i_\alpha, g_\beta}(F) + \Theta_{g_\alpha,h^i_\beta}(F) \big) = \pi \|F\|_{\textup{L}^4}^4.
\end{equation}
By the remark preceding the lemma, all terms on the left-hand side of \eqref{eq:telescoping} are nonnegative.
Therefore, \eqref{eq:telescoping} implies the inequalities
\[
\sum_{i=1}^d \Theta_{h^i_\alpha, g_\beta}(F) \lesssim \|F\|^4_{\textup{L}^4}, \quad\sum_{i=1}^d \Theta_{g_\alpha,h^i_\beta}(F) \lesssim \|F\|^4_{\textup{L}^4}.
\]
This establishes the claim of Lemma~\ref{lemma:mainlemma}, up to the verification of \eqref{eq:telescoping}.

To show the identity \eqref{eq:telescoping} we observe that by the fundamental theorem of calculus
\begin{align*}
& \sum_{i=1}^d \Big( \int_0^\infty (2\pi \alpha t \xi_i)^2 e^{-2\pi \|\alpha t \xi\|^2} e^{-2\pi \|\beta t \eta\|^2}\frac{dt}{t} +
\int_0^\infty e^{-2\pi \|\alpha t \xi\|^2} (2\pi \beta t \eta_i)^2 e^{-2\pi \|\beta t \eta\|^2} \frac{dt}{t} \Big) \\
& = \pi\int_0^\infty \Big(-t\partial_t \big(e^{-2\pi \|\alpha t \xi\|^2}e^{-2\pi \|\beta t \eta\|^2} \big) \Big) \frac{dt}{t} = \pi
\end{align*}
for any $\xi=(\xi_1,\ldots,\xi_d)\in\mathbb{R}^d$ and $\eta=(\eta_1,\ldots,\eta_d)\in\mathbb{R}^d$ such that $(\xi,\eta)\neq(0,0)$.
Using $\widehat{g}(\xi)=e^{-\pi\|\xi\|^2}$ and $\widehat{h^i}(\xi) = 2\pi\mathbbm{i} \xi_i \widehat{g}(\xi)$ this can be rewritten as
\begin{equation}\label{eq:pifourier}
 \sum_{i=1}^d \Big( \int_0^\infty |\widehat{h_{\alpha t}^i}(\xi)|^2 |\widehat{g_{\beta t}}(\eta)|^2 \frac{dt}{t}
+ \int_0^\infty |\widehat{g_{\alpha t}}(\xi)|^2 |\widehat{h_{\beta t}^i}(\eta)|^2 \frac{dt}{t} \Big) = \pi.
\end{equation}
Note that for real-valued Schwartz functions $\varphi$ and $\psi$ one has
\begin{align}\nonumber
& \int_{(\mathbb{R}^{d})^2} |\widehat{\psi_t}(\xi)|^2 |\widehat{\varphi_t}(\eta)|^2 e^{2\pi \mathbbm{i} ((x-y)\cdot\xi + (x'-y')\cdot\eta)} d\xi d\eta \\ \label{eq:inversionformula}
& = \int_{(\mathbb{R}^{d})^2} \psi_t(x-q)\psi_t(y-q)\varphi_t(x'-p)\varphi_t(y'-p) dp dq .
\end{align}
Indeed, for a function $\rho$ we denote $\widetilde{\rho}(s):=\overline{\rho(-s)}$, so that the Fourier transform of $\widetilde{\rho}$ is the complex conjugate of $\widehat{\rho}$. Equality \eqref{eq:inversionformula} follows by noticing that its right-hand side equals
\[
(\psi_t\ast\widetilde{\psi}_t) (x-y) (\varphi_t\ast\widetilde{\varphi}_t) (x'-y'),
\]
which in turn transforms into the left-hand side using the Fourier inversion formula and
\[
\widehat{\psi_t\ast\widetilde{\psi}_t} = |\widehat{\psi_t}|^2,\quad \widehat{\varphi_t\ast\widetilde{\varphi}_t} = |\widehat{\varphi_t}|^2.
\]

Now we multiply \eqref{eq:pifourier} by
\[
F(x,x') F(x,y') F(y,x') F(y,y') e^{2\pi \mathbbm{i} ((x-y)\cdot\xi + (x'-y')\cdot\eta)}
\]
and
integrate in $x,y,x',y'$ and $\xi,\eta$. Then we apply the inversion formula \eqref{eq:inversionformula}
twice, once with $(\psi,\varphi) = (h^i_\alpha,g_\beta)$ and once with $(\psi,\varphi) = (g_\alpha, h_\beta^i)$, and recall the definition \eqref{eq:defoftheta}. This gives
\begin{align*}
&\sum_{i=1}^d \big( \Theta_{h^i_\alpha, g_\beta}(F) + \Theta_{g_\alpha,h^i_\beta}(F) \big)\\
& = \pi \int_{(\mathbb{R}^d)^4} F(x,x') F(x,y') F(y,x') F(y,y') \delta_{(0,0)}(x-y,x'-y') dx dy dx' dy' = \pi \|F\|_{\textup{L}^4}^4.
\end{align*}
Here $\delta_{(0,0)}$ denotes the Dirac measure concentrated at the origin and it is a well-known fact that its Fourier transform is the function constantly equal to $1$ on the whole space $\mathbb{R}^d\times\mathbb{R}^d$.
\end{proof}

Observe that for $\nu>0$ and $x\in\mathbb{R}^d$ we have
\begin{equation}\label{eq:schwartzgauss}
(1+\|x\|)^{-\nu} \sim_\nu \int_{1}^{\infty} e^{-\pi\beta^{-2}\|x\|^2} \frac{d\beta}{\beta^{\nu+1}}.
\end{equation}
This formula is easily verified by continuity and considering the limiting behavior as $\|x\|\rightarrow\infty$, when the ratio of the two sides converges to
\[
\lim_{\|x\|\rightarrow\infty} \int_{1}^{\infty} e^{-\pi (\|x\|/\beta)^2} (\|x\|/\beta)^{\nu} \frac{d\beta}{\beta}
= \int_{0}^{\infty} e^{-\pi\alpha^2} \alpha^{\nu-1} d\alpha = \frac{1}{2} \pi^{-\nu/2} \Gamma\Big(\frac{\nu}{2}\Big) \in (0,\infty).
\]
It will be used in the proof of Theorem \ref{thm:analytictheorem} to gradually reduce to forms in which all bump functions are Gaussians or their derivatives. Gaussians possess several convenient algebraic properties, such as positivity, elementary tensor structure, and the fact that they relate differentiation to multiplication.

\begin{proof}[Proof of Theorem~\ref{thm:analytictheorem}]
By a density argument we can assume that $F$ and $G$ are real-valued Schwartz functions.
Substituting
\[ x' = x + y + u,\quad y' = x + y + v \]
and introducing the functions
\[
\widetilde{F}(a,b) := F(b-a,a),\quad \widetilde{G}(a,b) := G(a,b-a)
\]
the form in question can be written as
\begin{equation}\label{eq:rewrittenform}
\int_{(\mathbb{R}^d)^4} \widetilde{F}(y,x') \widetilde{G}(x,x') \widetilde{F}(y,y') \widetilde{G}(x,y') \widecheck{m}(x'-x-y,y'-x-y) dx dy dx' dy'.
\end{equation}
We need to bound its absolute value by a constant times
\[
\|F\|_{\textup{L}^4}^2 \|G\|_{\textup{L}^4}^2 = \|\widetilde{F}\|_{\textup{L}^4}^2 \|\widetilde{G}\|_{\textup{L}^4}^2.
\]
Let us henceforth omit writing tildes on the functions in \eqref{eq:rewrittenform}.
We will say that the form \eqref{eq:rewrittenform} is associated with the symbol $m$.

The first step is to decompose the kernel $\widecheck{m}$ into elementary tensors in the variables $x,y,x',y'$, which will allow for an application of the Cauchy-Schwarz inequality.

Let $\phi\in\mathcal{S}(\mathbb{R}^{2d})$ be a nonnegative radial function supported in the annulus $\{\tau\in \mathbb{R}^{2d} : 1\leq \|\tau\| \leq 2\}$ and not identically equal to $0$.
The constants in any estimates that follow are allowed to depend on $\phi$ without explicit mention.
Then
\begin{align*}
D := \int_0^\infty \phi(t\xi,t\eta) \|(t\xi,t\eta)\|^{2} e^{-\pi\|(t\xi,t\eta)\|^2} \frac{dt}{t}
\end{align*}
is the same constant for each $(\xi,\eta)\neq(0,0)$. Therefore, for each such pair $(\xi,\eta)$ we can write
\[
m(\xi,\eta) = D^{-1} \int_0^\infty m(\xi,\eta)\phi(t\xi,t\eta) \|(t\xi,t\eta)\|^{2} e^{-\pi\|(t\xi,t\eta)\|^2} \frac{dt}{t}.
\]
Using the identity
\begin{equation}\label{eq:elemid}
\| (\xi,\eta) \|^2 = \|\xi+\eta\|^2 - 2\xi\cdot\eta
\end{equation}
we can split further
\[
m = m^{[1]} + m^{[2]},
\]
where
\begin{align*}
m^{[1]}(\xi,\eta) & := D^{-1}\int_0^\infty m^{(t)}(t\xi,t\eta) \|t\xi+t\eta\|^2 e^{-\pi \|(t\xi,t\eta)\|^2} \frac{dt}{t}, \\
m^{[2]}(\xi,\eta) & := -2D^{-1}\int_0^\infty m^{(t)}(t\xi,t\eta) (t\xi\cdot t\eta) e^{-\pi \|(t\xi,t\eta)\|^2} \frac{dt}{t},
\end{align*}
and we have set
\[
m^{(t)}(\xi,\eta) := m(t^{-1}\xi,t^{-1}\eta)\phi(\xi,\eta).
\]
Now we separately study the forms associated with $m^{[1]}$ and $m^{[2]}$.

First we consider $m^{[1]}$. This is the easier term, as it vanishes on the plane $\xi+\eta=0$, which brings useful cancellation to our form. The remaining part of the proof related to $m^{[1]}$ can be compared with Sections 3 and 4 in \cite{DKST16}.

 Define the functions $\varphi^{(t)}$ and $\vartheta^{(i,t)}$ via their Fourier transforms as
\begin{align*}
\widehat{\varphi^{(t)}}(\xi,\eta) &:= m^{(t)}(\xi,\eta)e^{2\pi\xi\cdot\eta}, \\
\widehat{\vartheta^{(i,t)}}(\xi,\eta) &:= \widehat{\varphi^{(t)}}(\xi,\eta) \big( (\xi_i+\eta_i) e^{-2^{-1}\pi \|\xi+\eta\|^2} \big)^2.
\end{align*}
Observe that by $\|\xi+\eta\|^2 = \sum_{i=1}^d (\xi_i+\eta_i)^2$ and \eqref{eq:elemid} used in the exponent we have
\begin{equation}\label{eq:m1}
m^{[1]}(\xi,\eta) = D^{-1} \sum_{i=1}^d \int_0^\infty \widehat{\vartheta^{(i,t)}}(t\xi,t\eta) \frac{dt}{t}.
\end{equation}
By the Fourier inversion formula we can write
\begin{align}
& \vartheta_t^{(i,t)}(x'-x-y,y'-x-y) \nonumber \\
& = \int_{(\mathbb{R}^d)^2} \widehat{\varphi^{(t)}}(t\xi,t\eta) \big( (t\xi_i+t\eta_i) e^{- 2^{-1}\pi \|t\xi+t\eta\|^2} \big)^2 e^{2\pi\mathbbm{i} ( (x'-x-y)\cdot\xi + (y'-x-y)\cdot\eta )} d\xi d\eta \nonumber \\
& = -\frac{1}{2\pi^2}\int_{(\mathbb{R}^d)^2} \widehat{\varphi^{(t)}_{t}}(\xi,\eta) e^{2\pi\mathbbm{i} (x'\cdot\xi + y'\cdot\eta)}
\big( \widehat{h_{2^{-1/2}t}^i}(-\xi-\eta) \big)^2 e^{2\pi\mathbbm{i} (x\cdot(-\xi-\eta) + y\cdot(-\xi-\eta))} d\xi d\eta. \label{eq:subspace}
\end{align}
To pass from the second to the third line we have used $\widehat{h^i}(\xi) = 2\pi\mathbbm{i} \xi_i e^{-\pi\|\xi\|^2}$.
Using the definition of the Fourier transform, \eqref{eq:subspace} can be, up to a constant, viewed as the integral of the Fourier transform of the $4d$-dimensional function
\[
H(a,b,c,d):= {\varphi_t^{(t)}}(x'+a,y'+b) h_{2^{-1/2}t}^i(x+c) h_{2^{-1/2}t}^i(y+d)
\]
over a $2d$-dimensional subspace of $\mathbb{R}^{4d}$ parametrized by
\begin{equation}\label{eq:subspacepar}
\{(\xi,\eta,-\xi-\eta,-\xi-\eta): \xi, \eta \in \mathbb{R}^d\}.
\end{equation}
The integral of the Fourier transform of $H$ over the above mentioned subspace equals the integral of the function $H$ itself over the orthogonal complement of this subspace.
This fact can be found for instance in \cite{MS13}, and it is easily verified by performing an orthogonal change of variables, which rotates the two subspaces onto $2d$-dimensional coordinate planes in $\mathbb{R}^{2d}\times\mathbb{R}^{2d}$.
The orthogonal complement of \eqref{eq:subspacepar} can be parametrized by
\[
\{(-p-q,-p-q,-p,-q): p,q \in \mathbb{R}^d\}.
\]
Therefore, \eqref{eq:subspace} is a constant multiple of
\[
\int_{(\mathbb{R}^d)^2} {\varphi_t^{(t)}}(x'-p-q,y'-p-q)h^i_{2^{-1/2}t}(x-p)h^i_{2^{-1/2}t}(y-q)dpdq .
\]
Combining this with the decomposition of $m^{[1]}$ given in \eqref{eq:m1}, we see that the form associated with $m^{[1]}$ can be, up to a constant, recognized as
\begin{align}\label{eq:formspatialside}
\sum_{i=1}^d\int_0^\infty \int_{(\mathbb{R}^d)^6} & F(y,x')G(x,x')F(y,y')G(x,y') h_{2^{-1/2}t}^i(x-p) h_{2^{-1/2}t}^i(y-q) \nonumber \\
& \varphi^{(t)}_{t}(x'-p-q,y'-p-q) dx dy dx' dy' dp dq \frac{dt}{t}.
\end{align}
Note that vanishing of the multiplier on $\xi+\eta=0$ is crucial for the cancellation in $x$ and $y$ on the spatial side.

Now we are ready to proceed with an application of the Cauchy-Schwarz inequality.
We separate the functions in \eqref{eq:formspatialside} with respect to the variables $x,y$ and rewrite \eqref{eq:formspatialside} as
\begin{align*}
\sum_{i=1}^d\int_0^\infty \int_{(\mathbb{R}^d)^4} \Big( \int_{\mathbb{R}^d}&F(y,x')F(y,y')h_{2^{-1/2}t}^i(y-q) dy \Big ) \Big ( \int_{\mathbb{R}^d}G(x,x')G(x,y')h_{2^{-1/2}t}^i(x-p) dx \Big ) \\
& \varphi^{(t)}_{t}(x'-p-q,y'-p-q) dx' dy' dp dq \frac{dt}{t}.
\end{align*}
Then we apply the Cauchy-Schwarz inequality in $x',y',p,q,t$ and in $i$, after which it remains to bound
\begin{equation}\label{eq:formaftercs}
\sum_{i=1}^d \int_0^\infty \int_{(\mathbb{R}^d)^4} \Big( \int_{\mathbb{R}^d} F(y,x') F(y,y') h_{2^{-1/2}t}^i(y-q) dy \Big)^2 \big|\varphi^{(t)}_{t}(x'-p,y'-p)\big| dx' dy' dp dq \frac{dt}{t}
\end{equation}
and an analogous term involving the function $G$, which we omit. Note that we changed the variable $p$ to $p-q$ while simplifying \eqref{eq:formaftercs}. The next step is to reduce to Gaussians using the formula \eqref{eq:schwartzgauss}.
We have
\begin{align}
|\varphi^{(t)}(u,v)| & \lesssim_{d,(C_\kappa)} (1+\|(u,v)\|)^{-2d-1} \nonumber \\
& \sim_d \int_1^\infty e^{-\pi\beta^{-2}\|(u,v)\|^2} \frac{d\beta}{\beta^{2d+2}} = \int_1^\infty g_\beta(u)g_\beta(v) \frac{d\beta}{\beta^2}. \label{eq:varphidecay}
\end{align}
The first estimate above can be verified integrating by parts in the Fourier expansion of ${\varphi^{(t)}}$. It holds uniformly in $t>0$, with the implicit constant depending only on $d$ and the constants $C_\kappa$ appearing in \eqref{eq:symbolest}.
The second estimate above is simply \eqref{eq:schwartzgauss} for $x=(u,v)\in\mathbb{R}^{2d}$ and $\nu=2d+1$.
Substituting \eqref{eq:varphidecay} into \eqref{eq:formaftercs} and expanding out the square dominates \eqref{eq:formaftercs} by a constant multiple of
\[
\int_1^\infty \sum_{i=1}^d\Theta_{h^i_{\alpha}, g_{\beta}}(F) \frac{d\beta}{\beta^2},
\]
where $\alpha=2^{-1/2}$ and we recall the definition \eqref{eq:defoftheta}.
By Lemma \ref{lemma:mainlemma}, the last display is bounded by a constant multiple of
\[
\int_1^\infty \|F\|_{\textup{L}^4}^4 \frac{d\beta}{\beta^2} = \|F\|_{\textup{L}^4}^4,
\]
which concludes the proof of
boundedness of the form associated with $m^{[1]}$.

It remains to consider the form associated with the multiplier symbol $m^{[2]}$, which does not vanish on $\xi+\eta=0$. This part of the proof can be compared with Section 5 in \cite{DKST16}. In the one-dimensional case \cite{DKST16}, the multiplier was symmetrized to become constant on the axis $\xi+\eta=0$. Then that constant was subtracted from the multiplier and a lacunary decomposition with respect to the critical axis was performed. In the present higher-dimensional setting we also reduce the problem to parts vanishing on the problematic plane $\xi+\eta=0$. However, working with Gaussians allows us to do that by using several related algebraic identities.

Applying the Fourier inversion formula to $m^{(t)}$ and using $\|(\xi,\eta)\|^2 = \|\xi\|^2 + \|\eta\|^2$ in the exponent we can write
\[
m^{[2]}(\xi,\eta) = -2D^{-1}\int_{(\mathbb{R}^{d})^2} \int_0^\infty \widecheck{m^{(t)}}(u,v) (t\xi\cdot t\eta) e^{-\pi\|t\xi\|^2} e^{-2\pi\mathbbm{i} u\cdot t\xi} e^{-\pi\|t\eta\|^2} e^{-2\pi\mathbbm{i} v\cdot t\eta} \frac{dt}{t} dudv.
\]
Using $\xi\cdot\eta = \sum_{i=1}^{d} \xi_i \eta_i$ and taking the inverse Fourier transform of
\begin{align*}
\xi_i e^{-\pi\|t\xi\|^2-2\pi\mathbbm{i} u\cdot t\xi} \quad \textup{and} \quad \eta_i e^{-\pi\|t\eta\|^2-2\pi\mathbbm{i} v\cdot t\eta},
\end{align*}
we see that the form associated with $m^{[2]}$ can be, up to a constant, recognized as
\begin{align}
\sum_{i=1}^d\int_{(\mathbb{R}^{d})^2} \int_0^\infty & \widecheck{m^{(t)}}(u,v) \int_{(\mathbb{R}^{d})^4} F(y,x')G(x,x')F(y,y')G(x,y') \nonumber \\
& h_{t}^{i}(x'-x-y-tu)h_{t}^{i}(y'-x-y-tv)dxdydx'dy' \frac{dt}{t}dudv. \label{eq:formm2}
\end{align}

Now we would like to reduce the parameters $u$ and $v$ to only one parameter, which gives more symmetry. For this we first write \eqref{eq:formm2} as
\begin{align*}
\sum_{i=1}^d\int_{(\mathbb{R}^{d})^2} \int_0^\infty \widecheck{m^{(t)}}(u,v) \int_{(\mathbb{R}^{d})^2} &\Big( \int_{\mathbb{R}^{d}} F(y,x')G(x,x') h_{t}^{i}(x'-x-y-tu) dx' \Big )\\
&\Big( \int_{\mathbb{R}^{d}} F(y,y')G(x,y')h_{t}^{i}(y'-x-y-tv) dy' \Big) dx dy \frac{dt}{t}dudv.
\end{align*}
Then we use
\[
\big|\widecheck{m^{(t)}}(u,v)\big| \lesssim_{d,(C_\kappa)} (1+\|u\|)^{-d-1} (1+\|v\|)^{-d-1},
\]
which can be deduced analogously to the first estimate in \eqref{eq:varphidecay}, and apply the Cauchy-Schwarz inequality in $x,y$, and $t$. This yields
\begin{align*}
\sum_{i=1}^d \bigg( \int_{\mathbb{R}^d} (1+\|u\|)^{-d-1} \Big( & \int_0^\infty \int_{(\mathbb{R}^{d})^2} \Big( \int_{\mathbb{R}^d} F(y,x')G(x,x') \\
& h_{t}^{i}(x'-x-y-tu)dx'\Big)^2 dxdy\frac{dt}{t} \Big)^{1/2} du \bigg)^2.
\end{align*}
Indeed, note that after application of the Cauchy-Schwarz inequality the integrals in $u$ and $v$ have separated and they are equal. By another application of the Cauchy-Schwarz inequality, this time in $u$, we obtain
\begin{align}
\Big( \int_{\mathbb{R}}(1+\|u\|)^{-d-1} du\Big ) \sum_{i=1}^d \bigg( \int_{\mathbb{R}^d} (1+\|u\|)^{-d-1}
\int_0^\infty \int_{(\mathbb{R}^{d})^2} \Big( \int_{\mathbb{R}^d} F(y,x')G(x,x') & \nonumber \\
h_{t}^{i}(x'-x-y-tu)dx'\Big)^2 dxdy \frac{dt}{t} du \bigg) & . \label{eq:aftercs}
\end{align}
We evaluate the first integral in $u$ and dominate $(1+\|u\|)^{-d-1}$ in the second integral using \eqref{eq:schwartzgauss}, analogously to the domination in \eqref{eq:varphidecay}. Expanding the square in \eqref{eq:aftercs}, it then remains to bound
\begin{align}
\int_1^\infty \sum_{i=1}^d\int_{\mathbb{R}^d} g_\alpha(u) & \int_0^\infty \int_{(\mathbb{R}^{d})^4} F(y,x')G(x,x')F(y,y')G(x,y') \nonumber \\
& h_{t}^{i}(x'-x-y-tu)h^{i}_{t}(y'-x-y-tu)dxdydx'dy'du \frac{dt}{t} \frac{d\alpha}{\alpha^2}. \label{eq:formafterdom}
\end{align}
Note that it suffices to consider the expression in \eqref{eq:formafterdom} for each fixed $\alpha$ and obtain estimates that are uniform in $\alpha\geq 1$. Taking the Fourier transform of
\begin{align*}
\sum_{i=1}^d\int_{\mathbb{R}^d} g_\alpha(u) h_{t}^{i}(a-tu)h^{i}_{t}(b-tu) du
\end{align*}
in variable $(a,b)$ gives a constant multiple of
\[
(t\xi\cdot t\eta) e^{-\pi \|(t\xi,t\eta)\|^2} e^{-\pi\|\alpha t\xi + \alpha t\eta\|^2} .
\]
Therefore, the form \eqref{eq:formafterdom} for a fixed $\alpha$ is, up to a constant, associated with the symbol
\begin{align}\label{eq:multsymmetry}
\int_0^\infty (t\xi\cdot t\eta) e^{-\pi \|(t\xi,t\eta)\|^2} e^{-\pi\|\alpha t\xi + \alpha t\eta\|^2} \frac{dt}{t}.
\end{align}

Now that we have symmetrized the multiplier in $u$ and $v$, we go backwards: we again use the identity \eqref{eq:elemid} and write twice the expression in \eqref{eq:multsymmetry} as
\begin{align}
& \int_0^\infty \|t\xi + t\eta\|^2 e^{-\pi\|(t\xi,t\eta)\|^2} e^{-\pi\|\alpha t\xi + \alpha t\eta\|^2} \frac{dt}{t} \label{eq:mult1} \\
& - \int_0^\infty \|( t\xi, t\eta)\|^2 e^{-\pi\|(t\xi,t\eta)\|^2} e^{-\pi\|\alpha t\xi + \alpha t\eta\|^2} \frac{dt}{t}. \label{eq:mult2}
\end{align}

The term \eqref{eq:mult1} is easier to handle, and can be treated similarly as \eqref{eq:m1}. Indeed, note that \eqref{eq:mult1} can be further rewritten as
\begin{align}\label{eq:multfirstterm}
\alpha^{-2} \sum_{i=1}^d \int_0^\infty e^{-\pi\|(t\xi,t\eta)\|^2} \Big( (\alpha t\xi_i+\alpha t\eta_i) e^{-2^{-1}\pi \|\alpha t\xi+\alpha t\eta\|^2} \Big)^2 \frac{dt}{t}.
\end{align}
Performing the steps analogous to \eqref{eq:m1}--\eqref{eq:formspatialside} and observing $\alpha^{-2}\leq 1$, since we only consider $\alpha\geq 1$, it suffices to bound
\begin{align}\label{eq:multfirstterm2}
\sum_{i=1}^d\int_0^\infty \int_{(\mathbb{R}^d)^6} & F(y,x')G(x,x')F(y,y')G(x,y') h_{\alpha t}^i(x-p) h_{\alpha t}^i(y-q) \nonumber \\
& g_{2^{1/2}t}(x'-p-q)g_{2^{1/2}t}(y'-p-q) dx dy dx' dy' dp dq \frac{dt}{t}
\end{align}
uniformly in the parameter $\alpha$.
Separating the functions with respect to the variables $x,y$ and applying the Cauchy-Schwarz inequality analogously to \eqref{eq:formaftercs}, we estimate the last display by
\[
\Big ( \sum_{i=1}^d\Theta_{h^i_{\alpha},g_{\beta}}(F) \Big )^{1/2}\Big ( \sum_{i=1}^d\Theta_{h^i_{\alpha},g_{\beta}}(G) \Big )^{1/2} \lesssim \|F\|_{\textup{L}^4}^2\|G\|_{\textup{L}^4}^2,
\]
where $\beta=2^{1/2}$ and the last inequality follows from Lemma~\ref{lemma:mainlemma}.

It remains to consider the second term \eqref{eq:mult2}. Here we first use an integration by parts identity to transfer to a multiplier vanishing on the critical plane $\xi+\eta=0$. By the fundamental theorem of calculus we have
\begin{align}
& 2\pi \int_0^\infty \|(t\xi,t\eta)\|^2 e^{-\pi\|(t\xi,t\eta)\|^2} e^{-\pi\|\alpha t\xi + \alpha t\eta\|^2} \frac{dt}{t} \label{eq:tel} \\
& + 2\pi \int_0^\infty e^{-\pi\|(t\xi,t\eta)\|^2} \|\alpha t\xi + \alpha t\eta\|^2 e^{-\pi\|\alpha t\xi + \alpha t\eta\|^2} \frac{dt}{t} \label{eq:tel2} \\
& = \int_0^\infty \Big( -t\partial_t \big( e^{-\pi\|(t\xi,t\eta)\|^2} e^{-\pi\|\alpha t\xi + \alpha t\eta\|^2} \big) \Big) \frac{dt}{t} = 1 \nonumber
\end{align}
for $(\xi,\eta)\neq(0,0)$.
Since \eqref{eq:mult2} is up to a constant equal to the term in \eqref{eq:tel}, and the form associated with the constant symbol $1$ is trivially bounded,
it remains to consider the form associated with \eqref{eq:tel2}.
Note that it is analogous to \eqref{eq:mult1}, up to scaling in $\alpha$.

Expanding $\|\alpha t\xi + \alpha t\eta\|^2$ as in \eqref{eq:multfirstterm} and performing the steps analogous to \eqref{eq:m1}--\eqref{eq:formspatialside} we again arrive at the form \eqref{eq:multfirstterm2}, which is bounded by the preceeding discussion.
This finishes the proof.
\end{proof}


\section{The combinatorial part: Proof of Theorem~\ref{thm:cornerstheorem}}
\label{sec:combpart}
As mentioned in the introduction, our strategy of proof closely follows that in \cite{CMP15}. In our presentation we try to find a compromise between elaborating the key steps and avoiding repetition.

For a fixed $1<p<\infty$ the authors of \cite{CMP15} start by defining a measure supported on $S_\lambda = \{ s\in \mathbb{R}^d \colon \|s\|_{\ell^p} = \lambda \}$ that detects the correct size (of common differences or sides) in the $\ell^p$-norm. More precisely, for each $\lambda>0$ we define $\sigma_{\lambda}$ formally via the oscillatory integral
\[
\sigma_{\lambda}(s) := \lambda^{-d+p} \int_\mathbb{R} e^{2\pi \mathbbm{i} t(\|s\|_{\ell^p}^{p}-\lambda^{p})} dt,
\]
which turns out to be a measure that is mutually absolutely continuous with respect to the surface measure on $S_\lambda$. The form
\[
\mathcal{N}_{\lambda}(f) := \int_{(\mathbb{R}^d)^2} \int_{S_\lambda} f(x,y) f(x+s,y) f(x,y+s) d\sigma_{\lambda}(s) dx dy
\]
counts corners with respect to this measure. The main idea is to approximate $\mathcal{N}_{\lambda}(f)$ by a more convenient and smoother integral, defined using an appropriate Schwartz cutoff function, at which point we will be able to count the number of corners using a result from additive combinatorics.

Let $\psi\colon\mathbb{R}\to[0,1]$ be a Schwartz function such that $\widehat{\psi}$ is nonnegative and compactly supported, $\psi(0)=1$, and $\widehat{\psi}(1)>0$.
All constants in any estimates that follow are allowed to depend on $\psi$ and this dependence will be suppressed from the notation.

For $\varepsilon,\lambda>0$ define a function $\omega_{\lambda}^{\varepsilon}\colon\mathbb{R}^d\to\mathbb{C}$ that approximates the measure $\sigma_{\lambda}$ by
\[
\omega_{\lambda}^{\varepsilon}(s) := \lambda^{-d+p} \int_\mathbb{R} e^{2\pi \mathbbm{i} t(\|s\|_{\ell^p}^{p}-\lambda^{p})} \psi(\varepsilon\lambda^{p}t) dt
= \lambda^{-d} \varepsilon^{-1} \widehat{\psi}\Big(\varepsilon^{-1}\big(1-\|\lambda^{-1}s\|_{\ell^p}^p\big)\Big).
\]
It is a nonnegative, bounded, and compactly supported function (by our assumptions on $\widehat{\psi}$). Note that
\[
\omega_{\lambda}^{\varepsilon}(s) = \lambda^{-d} \omega_{1}^\varepsilon (\lambda^{-1}s),
\]
so the notation is still consistent with \eqref{eq:normalizationdef} from the previous section.
Moreover, in \cite{CMP15} it is shown that
\begin{equation}\label{eq:cepsdef}
\int_{\mathbb{R}^d} \omega_{\lambda}^{\varepsilon}(s) ds = c_1(\varepsilon) \int_{\mathbb{R}^d} \omega_{\lambda}^{1}(s) ds,
\end{equation}
where
\begin{equation}\label{eq:cepssim}
c_1(\varepsilon)\sim_{p,d} 1,
\end{equation}
for $0<\varepsilon<1/100d$.
Define
\[
\mathcal{M}_{\lambda}^{\varepsilon}(f) :=\int_{(\mathbb{R}^d)^3} f(x,y) f(x+s,y) f(x,y+s) \omega_{\lambda}^\varepsilon (s) ds dx dy.
\]

The first goal is to prove that $\mathcal{M}_{\lambda}^{1}(f)$ is large provided that the function $0\leq f\leq 1$ is dense.

\begin{proposition}\label{prop:m1large}
For any $1<p<\infty$, any positive integer $d$, any $0<\delta\leq 1$, and any $\lambda$ and $N$ satisfying $0<\lambda\leq N$ the following holds.
If $f\colon\mathbb{R}^d \times \mathbb{R}^d\to[0,1]$ is a measurable function supported in $[0,N]^d \times [0,N]^d$ and such that $\int_{[0,N]^{2d}} f \geq \delta N^{2d}$, then
\[
\mathcal{M}_{\lambda}^{1}(f) \gtrsim_{p,d,\delta} N^{2d}.
\]
\end{proposition}

When proving Proposition \ref{prop:m1large}, we borrow the following idea from \cite{CMP15}. In that paper the authors cut $\mathbb{R}^d$ into boxes that can be thought of as scaled images of $[0,1]^d$. On each of these boxes one then uses Roth's theorem for compact abelian groups \cite{Bou86}, the underlying group being the $d$-dimensional torus $\mathbb{T}^d$. We prove a similar result regarding corners in the unit box $[0,1]^d \times [0,1]^d$, which is equivalent to the same statement on $\mathbb{T}^d \times \mathbb{T}^d$.

\begin{lemma}\label{prop:cornersunitbox}
Let $0<\delta\leq 1$ and let $f\colon\mathbb{R}^d \times \mathbb{R}^d \to[0,1]$ be a measurable function supported in $[0,1]^d\times[0,1]^d$ and such that $\int_{[0,1]^{2d}}f \geq\delta$. Then
\[
\int_{([0,1]^d)^3} f(x,y) f(x+s,y) f(x,y+s) ds dx dy \gtrsim_{d,\delta} 1.
\]
\end{lemma}

Even though this lemma could be considered a quantitative variant of the well-known corners theorem \cite{AS74}, we could not find the exact reference to the corners theorem on compact abelian groups in the literature, so we deduce Lemma~\ref{prop:cornersunitbox} from its more familiar finitary formulation using the averaging trick of Varnavides \cite{Var59}.

\begin{proof}[Proof of Lemma~\ref{prop:cornersunitbox}]
Suppose that a positive integer $n$ is large enough so that each subset $S\subseteq\{0,1,\ldots,n-1\}^2$ of cardinality at least $(\delta/8)n^2$ must contain a corner. Such $n$ certainly exists by the result of Ajtai and Szemer\'{e}di \cite{AS74}, and by the theorem of Shkredov \cite{Shk06b} we even know that it is sufficient to take any $n\geq\exp (\exp(8/\delta)^c)$ for some absolute constant $c$.

First, we note that the set
\[
A := \Big\{(x,y)\in[0,1]^d\times[0,1]^d : f(x,y)\geq\frac{\delta}{2}\Big\}
\]
has measure at least $\delta/2$ and that $f\geq(\delta/2)\mathbbm{1}_A$, where $\mathbbm{1}_A$ denotes the indicator function of $A$. Therefore, it is enough to show
\begin{equation}\label{eq:cornersbound}
\int_{([0,1]^d)^3} \mathbbm{1}_A(x,y) \mathbbm{1}_A(x+s,y) \mathbbm{1}_A(x,y+s) ds dx dy \gtrsim_{d,\delta} 1.
\end{equation}

Take $\epsilon=\delta/16d$ and observe
\[
\fint_{(0,\epsilon/n]^d\times([0,1-\epsilon]^d)^2}\limits
\Big(\frac{1}{n^2} \sum_{i,j=0}^{n-1} \mathbbm{1}_{A}(u+it,v+jt)\Big) dt du dv
\geq |A\cap[\epsilon,1-\epsilon]^{2d}| \geq \frac{\delta}{2}-4d\epsilon = \frac{\delta}{4},
\]
where $\fint$ denotes the average value of the function on the given set. Defining
\[
T := \Big\{ (t,u,v)\in(0,\epsilon/n]^d\times[0,1-\epsilon]^d\times[0,1-\epsilon]^d : \frac{1}{n^2} \sum_{i,j=0}^{n-1} \mathbbm{1}_{A}(u+it,v+jt)\geq\frac{\delta}{8} \Big\},
\]
from the previous estimate we get
\begin{equation}\label{eq:slowerest}
|T| \geq \frac{\delta}{8} \Big(\frac{\epsilon}{n}\Big)^d (1-\epsilon)^{2d} \gtrsim_{d,\delta} 1.
\end{equation}
For each triple $(t,u,v)\in T$ we consider the set
\[
B_{t,u,v} := \big\{(i,j)\in\{0,1,\ldots,n-1\}^2 : (u+it,v+jt)\in A\big\}.
\]
Since $B_{t,u,v}$ contains at least $(\delta/8)n^2$ elements, by the choice of $n$ we conclude that $B_{t,u,v}$ must contain a corner $(i,j)$, $(i+k,j)$, $(i,j+k)$, which can be rewritten as
\[
\sum_{\substack{i,j,k\in\{0,1,\ldots,n-1\}\\ k\geq 1,\ i+k,j+k\leq n-1}}
\mathbbm{1}_{A}(u+it,v+jt) \mathbbm{1}_{A}(u+it+kt,v+jt) \mathbbm{1}_{A}(u+it,v+jt+kt) \geq 1.
\]
Integrating this over $(t,u,v)\in T$, using \eqref{eq:slowerest}, and changing variables to
\[
x=u+it,\quad y=v+jt,\quad s=kt
\]
we obtain
\begin{equation}\label{eq:cornersbound2}
\sum_{\substack{i,j,k\in\{0,1,\ldots,n-1\}\\ k\geq 1,\ i+k,j+k\leq n-1}}
\frac{1}{k^d} \int_{[0,k\epsilon/n]^d\times([0,1]^d)^2} \mathbbm{1}_A(x,y) \mathbbm{1}_A(x+s,y) \mathbbm{1}_A(x,y+s) ds dx dy \gtrsim_{d,\delta} 1.
\end{equation}
It remains to observe that the left-hand side of \eqref{eq:cornersbound2} is at most $n^3$ times the left-hand side of \eqref{eq:cornersbound}, recalling that $n$ can be taken to be a function depending only on $\delta$.
\end{proof}

\begin{proof}[Proof of Proposition~\ref{prop:m1large}]
It is straightforward to adapt the proof of the analogous proposition from \cite{CMP15} in the language of corners, replacing Roth's theorem on compact abelian groups \cite{Bou86} with Lemma~\ref{prop:cornersunitbox}.
\end{proof}

Our next aim is to prove that $\mathcal{N}_{\lambda}$ and $\mathcal{M}_{\lambda}^{\varepsilon}$ are in some sense close to each other.

\begin{proposition}\label{prop:nmclose}
For any $p\in(1,2)\cup(2,\infty)$ there exists $\gamma_p>0$ such that for any positive integer $d$, any $0<\varepsilon<1$, and any $\lambda$ and $N$ satisfying $0<\lambda\leq N$ the following holds.
If $f\colon\mathbb{R}^d\times\mathbb{R}^d\to[-1,1]$ is a measurable function supported in $[0,N]^d \times [0,N]^d$, then
\[
| \mathcal{N}_{\lambda}(f) - \mathcal{M}_{\lambda}^{\varepsilon}(f) | \lesssim_{p,d} \varepsilon^{d\gamma_p -1} N^{2d}.
\]
\end{proposition}

The proof of this proposition uses \emph{uniformity norms} or the \emph{$\textup{U}^k$-norms}, which Gowers introduced in his work on Szemer\'edi's theorem on the integers \cite{Gow98},\cite{Gow01}. For a measurable function $f\colon\mathbb{R}^d\to\mathbb{C}$ we define the \emph{Gowers uniformity norm on $\mathbb{R}^d$ of degree $k$} by
\[
\|f\|_{\textup{U}^k}^{2^k} := \int_{(\mathbb{R}^d)^{k+1}} \Delta_{h_1} \cdots \Delta_{h_k} f(x) dx dh_1 \cdots dh_k,
\]
where $\Delta_{h} f(x) := f(x) \overline{f(x+h)}$. A linear change of variables immediately yields
\begin{equation}\label{eq:ukscaling}
\|f_t\|_{\textup{U}^k} = t^{-d(1-(k+1)/2^k)} \|f\|_{\textup{U}^k}
\end{equation}
for all $t>0$.

\begin{proof}[Proof of Proposition~\ref{prop:nmclose}]
By a density argument we can assume that $f$ is continuous.

From the discussion preceding the proof we know how $\| \omega_{\lambda}^{\eta} - \omega_{\lambda}^{\varepsilon} \|_{\textup{U}^3}$ is defined for $0<\eta<\epsilon$.
The authors of \cite{CMP15} also give a meaning to $\| \sigma_{\lambda} - \omega_{\lambda}^{\varepsilon} \|_{\textup{U}^3}$ by interpreting it as the limit
$\lim_{\eta\rightarrow 0^+} \| \omega_{\lambda}^\eta - \omega_{\lambda}^{\varepsilon} \|_{\textup{U}^3}$,
which is justified by the facts that $(\omega_{\lambda}^{\eta})_{\eta>0}$ is a Cauchy net in the $\textup{U}^3$-norm and that it converges vaguely to $\sigma_{\lambda}$ as $\eta\rightarrow 0^+$.
Moreover, in \cite{CMP15} it is shown that for any $1<p<\infty$, $p\neq 2$ there exists a constant $\gamma_p>0$ such that for each integer $d$, any $0<\varepsilon<1$, and any $\lambda>0$ one has
\begin{equation}\label{eq:propaux1}
\| \sigma_{\lambda} - \omega_{\lambda}^{\varepsilon} \|_{\textup{U}^3} \lesssim_{p,d} \lambda^{-d/2} \varepsilon^{d \gamma_p - 1}.
\end{equation}
Indeed, it suffices to work out the case $\lambda=1$ and the general result follows from the scaling identity \eqref{eq:ukscaling}.

On the other hand, by applying the Cauchy-Schwarz inequality three times, for an arbitrary measurable function $g\colon\mathbb{R}^d\to\mathbb{R}$ supported in a constant dilate of the cube $[-\lambda,\lambda]^d$ one obtains
\[
\Big| \int_{(\mathbb{R}^d)^3} f(x,y) f(x+s,y) f(x,y+s) g(s) ds dx dy \Big| \lesssim N^{2d} \lambda^{d/2} \|g\|_{\textup{U}^3},
\]
the so-called generalized von Neumann's theorem, this time for corners. Setting $g = \omega_{\lambda}^{\eta} - \omega_{\lambda}^{\varepsilon}$ and letting $\eta\rightarrow 0^+$ we get
\begin{equation}\label{eq:propaux2}
|\mathcal{N}_{\lambda} (f) - \mathcal{M}_{\lambda}^{\varepsilon}(f) | \lesssim N^{2d} \lambda^{d/2} \| \sigma_{\lambda} - \omega_{\lambda}^{\varepsilon} \|_{\textup{U}^3}.
\end{equation}
It remains to combine \eqref{eq:propaux1} and \eqref{eq:propaux2} and the claim follows.
\end{proof}

As the final step, we use Theorem \ref{thm:analytictheorem} to connect $\mathcal{M}_{\lambda}^{1}(f)$ and $\mathcal{M}_{\lambda}^{\varepsilon}(f)$, where $\lambda$ goes through a sequence of scalars. Motivated by \cite{CMP15}, we define
\[
k_{\lambda}^\varepsilon(s) := \omega_{\lambda}^{\varepsilon}(s) - c_1(\varepsilon) \omega_{\lambda}^{1}(s),
\]
which is consistent with the notation \eqref{eq:normalizationdef}, and also set
\[
\mathcal{E}_{\lambda}^{\varepsilon}(f) := \mathcal{M}_{\lambda}^{\varepsilon}(f) - c_1(\varepsilon) \mathcal{M}_{\lambda}^{1}(f),
\]
where $c_1(\varepsilon)$ is the constant from \eqref{eq:cepsdef}. We prove the following result.

\begin{proposition}\label{prop:ebounded}
Let $0<\varepsilon < 1$, and let $d$ and $J$ be positive integers. Suppose that $\lambda_1<\lambda_2<\cdots<\lambda_J$ are positive numbers such that $\lambda_{j+1}/\lambda_j \geq 2$ for each $1\leq j\leq J-1$. If $f\colon\mathbb{R}^d\times\mathbb{R}^d\to[-1,1]$ is a measurable function supported in $[0,N]^d \times [0,N]^d$, then
\begin{equation}\label{eq:eupperbound}
\sum_{j=1}^{J} |\mathcal{E}_{\lambda_j}^{\varepsilon}(f)|^2 \lesssim_{d,\varepsilon} N^{4d}.
\end{equation}
\end{proposition}

\begin{proof}[Proof of Proposition~\ref{prop:ebounded}]
Using the definition of $\mathcal{E}_{\lambda}^{\varepsilon}(f)$ and applying the Cauchy-Schwarz inequality we estimate:
\begin{align*}
\sum_{j=1}^{J} |\mathcal{E}_{\lambda_j}^{\varepsilon}(f)|^2
& \leq \sum_{j=1}^{J} \Big( \int_{(\mathbb{R}^d)^2} f(x,y) \Big| \int_{\mathbb{R}^d} f(x+s,y) f(x,y+s) k_{\lambda_j}^\varepsilon(s) ds \Big| dx dy \Big)^2 \\
& \leq \|f\|_{\textup{L}^2}^2 \sum_{j=1}^{J} \int_{(\mathbb{R}^d)^2} \Big( \int_{\mathbb{R}^d} f(x+s,y) f(x,y+s) k_{\lambda_j}^\varepsilon (s) ds \Big)^2 dx dy \\
& = \|f\|_{\textup{L}^2}^2 \int_{(\mathbb{R}^d)^4} f(x+u,y) f(x,y+u) f(x+v,y) f(x,y+v) K(u,v) du dv dx dy,
\end{align*}
where we have written $K(u,v) := \sum_{j=1}^{J} k_{\lambda_j}^\varepsilon(u) k_{\lambda_j}^\varepsilon(v)$. It was verified in \cite{CMP15} that $m=\widehat{K}$ satisfies the symbol estimates \eqref{eq:symbolest} with the constants $C_\kappa$ depending only on $\kappa$, $d$, and $\varepsilon$. Therefore, Theorem~\ref{thm:analytictheorem} can be applied and yields
\begin{equation*}
\sum_{j=1}^{J} |\mathcal{E}_{\lambda_j}^{\varepsilon}(f)|^2
\lesssim_{d,\varepsilon} \|f\|_{\textup{L}^2}^2 \|f\|_{\textup{L}^4}^4 \leq N^{4d}. \qedhere
\end{equation*}
\end{proof}

We now deduce Theorem~\ref{thm:cornerstheorem} from Propositions~\ref{prop:m1large}, \ref{prop:nmclose}, and \ref{prop:ebounded}.

\begin{proof}[Proof of Theorem~\ref{thm:cornerstheorem}]
We argue by contradiction. Recall the constant $\gamma_p$ from Proposition~\ref{prop:nmclose}. If Theorem~\ref{thm:cornerstheorem} does not hold, then for some $1<p<\infty$, $p\neq 2$ and some $d>1/\gamma_p$ there exists a measurable set $A\subseteq\mathbb{R}^{2d}$ with $\overline{\delta}_{2d}(A)>0$ such that the side lengths of corners in $A$, measured in the $\ell^p$-norm, avoid values from some positive sequence $(\lambda_j)_{j=1}^{\infty}$ converging to $+\infty$. We can sparsify this sequence if necessary, so that it satisfies $\lambda_{j+1}/\lambda_j \geq 2$ for each index $j$. Fix any positive integer $J$. It will be enough to consider finitely many scales $\lambda_1<\cdots<\lambda_J$.

By the definition of upper Banach density, for any fixed $0< \delta < \overline{\delta}_{2d}(A)$, there exists a number $N\geq \lambda_J$ for which there is $x_N\in\mathbb{R}^{2d}$ such that $|A\cap(x_N+[0,N]^{2d})| \geq \delta N^{2d}$. If we denote $A_N:=(-x_N+A)\cap[0,N]^{2d}$, then $A_N$ is a measurable subset of $[0,N]^{2d}$ with measure at least $\delta N^{2d}$ such that the side length of any corner inside $A_N$ avoids the values $\lambda_1,\ldots,\lambda_J$. The latter property immediately implies that $\mathcal{N}_{\lambda_j}(\mathbbm{1}_{A_N})=0$ for each $1\leq j\leq J$.

Let us apply the three auxiliary propositions with $f=\mathbbm{1}_{A_N}$, recalling that this is the indicator function of $A_N$. Note that $\lim_{\varepsilon\rightarrow0^+}\varepsilon^{d\gamma_p -1}=0$ by our choice of $d$. Therefore, if $\varepsilon>0$ is taken small enough (depending on $p,d,\delta$), then \eqref{eq:cepssim} and Propositions~\ref{prop:m1large} and \ref{prop:nmclose} give
\[
|\mathcal{E}_{\lambda_j}^{\varepsilon}(f)|
\geq c_1(\varepsilon) \mathcal{M}_{\lambda_j}^{1}(f) - |\mathcal{N}_{\lambda_j}(f) - \mathcal{M}_{\lambda_j}^{\varepsilon}(f)|
\gtrsim_{p,d,\delta} N^{2d}.
\]
Consequently,
\begin{equation}\label{eq:elowerbound}
\sum_{j=1}^{J} |\mathcal{E}_{\lambda_j}^{\varepsilon}(f)|^2 \gtrsim_{p,d,\delta} J N^{4d}.
\end{equation}
Combining \eqref{eq:eupperbound} and \eqref{eq:elowerbound}, and dividing by $N^{4d}$, we conclude that $J\lesssim_{p,d,\delta,\varepsilon} 1$.
Recalling that $J$ could have been taken arbitrarily large we arrive at the contradiction.
\end{proof}

It is worth observing that a variant of the bound \eqref{eq:eupperbound} with a constant $o(J)$ on the right-hand side would have been sufficient. It is plausible that such a bound could be easier to establish than the uniform estimates in Theorem~\ref{thm:analytictheorem} and Proposition~\ref{prop:ebounded}.
However, the scales $\lambda_1<\cdots<\lambda_J$ in such a bound must comprise an arbitrary lacunary sequence. For instance, obtaining a $o(J)$ estimate for consecutive dyadic scales $\lambda_j=2^j$ is considerably easier; compare with the closing remarks in the next section.


\section{Remarks on possible generalizations}
\label{sec:genremarks}
It is natural to ask if the generalization of Theorem~\ref{thm:cmptheorem} holds for $k$-term arithmetic progressions in $\mathbb{R}^d$,
\[
x,\ x+s,\ x+2s,\ \ldots,\ x+(k-1)s,
\]
and if Theorem~\ref{thm:cornerstheorem} extends to the generalized $k$-element corners in $(\mathbb{R}^d)^{k-1}$,
\[
(x_1,x_2,\ldots,x_{k-1}),\ (x_1+s,x_2,\ldots,x_{k-1}),\ (x_1,x_2+s,\ldots,x_{k-1}),\ \ldots,\ (x_1,x_2,\ldots,x_{k-1}+s).
\]
The result that any positive upper density subset of $\mathbb{Z}^{k-1}$ has to contain a nontrivial $k$-element corner is popularly known as the \emph{multidimensional Szemer\'{e}di theorem} and was first shown by Furstenberg and Katznelson \cite{FK78}.

The following proposition is a straightforward generalization of the aforementioned counterexample of Bourgain.
It prohibits $p$ from taking any integer value less than $k$.

\begin{proposition}\label{prop:counterexample}
Let $d,k,p$ be positive integers such that $p\leq k-1$. There exists a measurable set $A\subseteq \mathbb{R}^d$ of positive upper Banach density such that no $\lambda_0>0$ satisfies the property that for each $\lambda\geq\lambda_0$ one can find a $k$-term arithmetic progression $x$, $x+s$, \ldots, $x+(k-1)s$ in $A$ with $\|s\|_{\ell^p} = \lambda$.
\end{proposition}

\begin{proof}[Proof of Proposition~\ref{prop:counterexample}]
Take $x=(x_1,\ldots,x_d)$ and $s=(s_1,\ldots,s_d)$ in $\mathbb{R}^d$ such that $x+js$ has nonnegative coordinates for $j=0,\ldots,p$ and observe the identity
\begin{equation}\label{eq:binomialid}
\sum_{j=0}^{p} (-1)^{p-j} {p\choose j} \|x+js\|_{\ell^p}^p = p! \|s\|_{\ell^p}^p.
\end{equation}
It is a direct consequence of the scalar identity
\[
\sum_{j=0}^{p} (-1)^{p-j} {p\choose j} (\alpha+j\beta)^l = \begin{cases} 0 & \text{for } l=0,1,\ldots,p-1, \\ p! \beta^p & \text{for } l=p, \end{cases}
\]
applied with $l=p$, $\alpha=x_i$, $\beta=s_i$, $i=1,\ldots,d$, which in turn can be easily established by induction on $p$.

Led by the example for three-term progressions, we define
\begin{equation}\label{eq:longercounter}
A := \bigcup_{n=1}^{\infty} \big\{ x\in[0,\infty)^d : n-2^{-p-2} \leq \|x\|_{\ell^p}^{p} \leq n+2^{-p-2} \big\}.
\end{equation}
As before, the set $A$ is made up of parts of spherical shells, but this time with respect to the $\ell^p$-norm.
It is easy to see that it still satisfies $\overline{\delta}_d(A)>0$.

Suppose that $x,s\in\mathbb{R}^d$ are such that $x+js\in A$ for $j=0,1,\ldots,k-1$. We only need to consider the first $p+1$ terms of this progression. By construction, $\|x+js\|_{\ell^p}^p$ differs from some positive integer $n_j$ by at most $2^{-p-2}$. From \eqref{eq:binomialid} we see that $p!\|s\|_{\ell^p}^p$ differs from the integer $\sum_{j=0}^{p}(-1)^{p-j}{p \choose j}n_j$ by at most
\[
\sum_{j=0}^{p} {p\choose j} \big| \|x+js\|_{\ell^p}^p - n_j \big| \leq \sum_{j=0}^{p} {p \choose j} 2^{-p-2} = \frac{1}{4}.
\]
Consequently, $\|s\|_{\ell^p}$ cannot attain values in the set
\begin{equation}\label{eq:gapexceptions}
\bigcup_{n=1}^{\infty} \bigg( \Big(\frac{4n-3}{4p!}\Big)^{1/p}, \Big(\frac{4n-1}{4p!}\Big)^{1/p} \bigg),
\end{equation}
which is unbounded from above.
\end{proof}

Example \eqref{eq:longercounter} from the previous proof also leads to a counterexample for generalized corners, by considering
\[
\widetilde{A} := \big\{(x_1,x_2,\ldots,x_{k-1})\in(\mathbb{R}^d)^{k-1} : x_1 + 2x_2 + \cdots + (k-1) x_{k-1}\in A\big\}.
\]
Once again, this set has $\overline{\delta}_{(k-1)d}(\widetilde{A})>0$, but the $\ell^p$-norm of the side $s$ of each $k$-element corner in $\widetilde{A}$ cannot belong to the set \eqref{eq:gapexceptions}.

There is still a chance that Theorems~\ref{thm:cmptheorem} and \ref{thm:cornerstheorem} generalize to $k\geq 4$ and any $1\leq p\leq\infty$ other than $1,2,\ldots,k-1$, and $\infty$. However, the corresponding analogs of Theorem~\ref{thm:analytictheorem} would involve operators of complexity similar as to the so-called \emph{multilinear and simplex Hilbert transforms} (see \cite{KTZ15},\cite{Tao15},\cite{ZKr15}), for which no $\textup{L}^p$-boundedness results are known at the time of writing.
An encouraging sign is that the papers \cite{Tao15} and \cite{ZKr15} establish estimates for the truncations of these operators with constants $o(J)$ in the number of consecutive dyadic scales $J$, while \cite{DKT16} improves this bound to $J^{1-\epsilon}$ for some $\epsilon>0$. As one needs to consider arbitrary (and not only consecutive dyadic) scales for the intended application, we believe that generalizations to large values of $k$ are still out of reach of the currently available techniques.


\section*{Acknowledgments}
P. D. is supported by the Hausdorff Center for Mathematics.
V. K. is supported in part by the Croatian Science Foundation under the project 3526.
L. R. is supported by the School of Mathematics at the University of Bristol.
P. D. and V. K. are partially supported by the bilateral DAAD-MZO grant \emph{Multilinear singular integrals and applications}.
The authors would like to thank Christoph Thiele, Julia Wolf, and the members of the HARICOT seminar for useful discussions and helpful comments.


\end{document}